\documentclass[11pt]{amsart}
\usepackage[colorlinks=true, pdfstartview=FitV, linkcolor=blue, 
citecolor=blue]{hyperref}

\usepackage{amsmath,amssymb,bbm,amscd}
\usepackage{graphicx}
\usepackage{a4wide}

\theoremstyle{plain}
\newtheorem{theorem}{Theorem}[section]
\newtheorem{prop}[theorem]{Proposition}
\newtheorem{lemma}[theorem]{Lemma}
\newtheorem{coro}[theorem]{Corollary}
\newtheorem{fact}[theorem]{Fact}

\theoremstyle{definition}
\newtheorem{example}[theorem]{Example}
\newtheorem{remark}[theorem]{Remark}
\newtheorem{definition}[theorem]{Definition}

\numberwithin{equation}{section}
 
\newcommand{\ts}{\hspace{0.5pt}}
\newcommand{\nts}{\hspace{-0.5pt}}

\DeclareMathOperator{\dens}{\mathrm{dens}}
\DeclareMathOperator{\card}{\mathrm{card}}
\DeclareMathOperator{\cent}{\mathrm{cent}}
\DeclareMathOperator{\homeo}{\mathrm{Homeo}}
\DeclareMathOperator{\norm}{\mathrm{norm}}

\DeclareMathOperator{\id}{\mathrm{id}}

\newcommand{\cB}{\mathcal{B}}
\newcommand{\cC}{\mathcal{C}}

\newcommand{\cH}{\mathcal{H}}

\newcommand{\cR}{\mathcal{R}}
\newcommand{\cS}{\mathcal{S}}

\newcommand{\cL}{\mathcal{L}}

\newcommand{\AAA}{\mathbb{A}}
\newcommand{\ZZ}{\mathbb{Z}\ts}

\newcommand{\NN}{\mathbb{N}}
\newcommand{\QQ}{\mathbb{Q}}

\newcommand{\XX}{\mathbb{X}}

\newcommand{\defeq}{\mathrel{\mathop:}=}

\newcommand{\exend}{\hfill $\Diamond$}

\newcommand{\bs}[1]{\boldsymbol{#1}}

\begin{document}

\title[On the full centraliser of Erd\H{o}s $\cB$-free shifts]{On
the full centraliser of Erd\H{o}s $\cB$-free shifts}

\author{Michael Baake}

\author{Neil Ma\~{n}ibo}

\address{Fakult\"{a}t f\"{u}r Mathematik,
  Universit\"{a}t Bielefeld,\newline \hspace*{\parindent}Postfach
  100131, 33501 Bielefeld, Germany}
\email{$\{$mbaake,cmanibo$\}$@math.uni-bielefeld.de}

\makeatletter
\@namedef{subjclassname@2020}{%
  \textup{2020} Mathematics Subject Classification}
\makeatother

\keywords{Number-theoretic shift spaces, symmetries, full centraliser
  and normaliser}

\subjclass[2020]{37B10, 11A07, 52C23}

\begin{abstract}
  The sets of $\cB$-free integers are considered with respect to
  (reversing) symmetries. It is well known that, for a large class of
  them, the centraliser of the associated $\cB$-free shift (otherwise
  known as its automorphism group) is trivial. We extend this result
  to the full centraliser, which effectively means to show that all
  self-homeomorphisms of the $\cB$-free shift that commute with some
  power of the shift are shifts themselves. This also leads to the
  result that the full normaliser agrees with the normaliser for this
  class, which is the semi-direct product of the centraliser with the
  cyclic group of order two generated by reflection.
\end{abstract}

\maketitle
\thispagestyle{empty}

\section{Introduction}\label{sec:intro}

Given a dynamical system under the action of an invertible mapping
(which generates a group isomorphic with $\ZZ$), determining its
symmetries and reversing symmetries is a natural step in its analysis;
see \cite{Lamb-3} and references therein for background. While this is
a standard tool in concrete dynamical systems, there is also a
counterpart in abstract dynamics. Here, given a \emph{topological
  dynamical system} (TDS), written as $(\XX, \ZZ)$, one is interested
in its automorphism group, which is the centraliser of
$\ZZ = \langle S \ts \rangle$ in the group $\cH \defeq \homeo (\XX)$
of self-homeomorphisms of $\XX$. This is widely studied in symbolic
dynamics, with some focus on low-complexity systems with centralisers
that are virtually $\ZZ$.

For both types of dynamical systems, one is also interested in the
presence of \emph{reversibility} \cite{Lamb-1, BR, Goodson, GdJLR},
where one asks whether the generator $S$ of $\ZZ$ is conjugated into
its inverse by an element of $\cH$. In concrete dynamics, it has long
been known that $k$-symmetries (respectively $k$-reversing symmetries,
with $k\in\NN$) are also important \cite{QL}, which refer to elements
from $\cH$ that commute with $S^k$, but possibly not with smaller
powers (respectively conjugate $S^k$ into its inverse). The
consideration of genuine $k$-symmetries has recently made its
appearance in symbolic dynamics, by determining not only the
automorphism group, but also its \emph{stabilised} extension; see
\cite{HKS, Schmieding, EJ,Jones-Baro, ES} and references therein.

On the other end of the (complexity) spectrum, there are interesting
dynamical systems of number-theoretic origin, such as the $k$-free
integers $V_{k} \subset \ZZ$, for some $2\leqslant k \in\NN$.  If
$\XX$ denotes the orbit closure of the point set $V_{k}$ in the
standard local topology, one is back to a TDS of the form
$(\XX, \ZZ)$, which can equivalently be viewed as a subshift of
$\{ 0, 1\}^{\ZZ}$.  Such shifts have interesting properties, as also
known from the visible (or primitive) lattice points \cite{Apo, BMP}
and various generalisations to $\cB$-free lattice systems \cite{BH,
  BBHLN}. The latter, in turn, are also generalisations of the
recently much-studied $\cB$-free integers \cite{Abda,DKKL}. In fact,
via the Minkowski embedding of rings of algebraic integers,
previously-studied extensions to number fields \cite{CV, BBHLN} can
also be viewed as $\cB$-free lattice systems.

It is an interesting general observation that such systems can also be
described in the setting of weak model sets \cite{TAO,BHS,Keller,KKL},
which builds on the pioneering work of Meyer \cite{Meyer} and gives
rapid access to various spectral and dynamical properties of such
systems \cite{BHS,HR}. Among these results is the statement that the
dynamical spectrum is pure point, though no eigenfunction (except the
trivial one) has a continuous representative, and closed forms for the
spectrum and for the topological entropy of such shift spaces.

Here, we are interested in another topological invariant of $k$-free
and, more generally, $\cB$-free shifts, namely their automorphism group
and its stabilised extension \cite{HKS}. We prefer the terms
\emph{centraliser} and \emph{full centraliser} (of the shift action
relative to $\cH$) for these two groups, thus fitting it into the
setting of reversing and extended symmetries \cite{Lamb-1, Baa, BRY},
which are given by the normaliser. The latter, due to an obvious
reflection symmetry, will be the semi-direct product of the
centraliser with the group $C_2$ generated by the reflection in the
origin. \vspace*{1mm}

Our main result (Theorem~\ref{thm:full-cent}) states that
non-degenerate $\cB$-free shifts of Erd\H{o}s type (to properly be
defined in Definition~\ref{def:simple}) do not only have a trivial
centraliser, but also a trivial full centraliser. At first, this
sounds a bit surprising because these systems have positive
topological entropy, and are thus somewhat opposite to low-complexity
shifts, though trivial centralisers occur for many dynamical systems
of number-theoretic origin \cite{BBHLN, BBN, cyclo}.

In the realm of positive entropy, it was recently shown
\cite[Thm.~2]{Schmieding} that two non-trivial mixing shifts of finite
type, $\XX$ and $\XX'$, with isomorphic full centralisers must have
rationally dependent topological entropies, that is,
$h_{\text{top}}(\XX)/h_{\text{top}}(\XX')\in \QQ$. This no longer
holds in the context of $\cB$-free shifts, see Eq.~\ref{eq:not-rat},
which is just another indication that such shifts are rather different
from shifts of finite type. The normaliser does not present further
constraints, so cannot restore the rationality property. Especially in
higher dimensions, the normaliser can capture important (and often
obvious) algebraic or geometric symmetries of the shift space that are
invisible for the centraliser; compare \cite{BRY, BBHLN}.

Just as positive-entropy shifts can have trivial centralisers,
zero-entropy shifts can also have large normalisers; see \cite{CP} for
a family of higher-dimensional (substitution) shifts that have this
property. It is fair to say that the spectrum of constraints that lead
to these phenomena is not yet well understood, which was one
motivation to consider the full centraliser for some number-theoretic
shifts, where positive entropy might manifest in some way. To the best
of our knowledge, the family of non-degenerate Erd\H{o}s $\cB$-free
shifts treated below is the first example with trivial full
centraliser while having positive topological entropy; compare
\cite{HKS,EJ,Schmieding}. The corresponding result also holds for the
(full) normaliser, which is $\ZZ\rtimes C_2$ for this class
(Corollary~\ref{coro:main}).

\section{Preliminaries and setup}\label{sec:initial}

Given a collection $\cB = \{ b^{}_{1}, b^{}_{2}, \dots \}$ of
positive integers, the set
\begin{equation}\label{eq:def-bfree}
  V^{}_{\cB} \, \defeq \, \{ n \in \ZZ : n \text{ is not
    divisible by any }  b \in \cB \ts \}
\end{equation}
is called the \emph{set of $\cB$-free numbers}. Since $1 \in \cB$
implies $V^{}_{\cB} = \varnothing$, we assume $1\notin\cB$. When $\cB$
is a finite set, $V^{}_{\cB}$ is periodic and thus not of particular
interest. We thus also assume that $\cB = \{ b_i : i \in \NN\} $ is
infinite and ordered increasingly, so $b_i < b_{i+1}$ holds for all
$i\in\NN$. Among the uncountably many choices, we are interested in
one particular class as follows.

\begin{definition}\label{def:simple}
  The set $\cB$ is said to be of \emph{Erd\H{o}s type}, if
  the following conditions are satisfied.
\begin{enumerate}\itemsep=2pt
\item $1 < b^{}_{1}$ and $b^{}_{i} < b^{}_{i+1}$ for all $i\in\NN$.
\item The $b_{i}$ are pairwise coprime.
\item One has $\sum_{i \in \NN} 1/b_{i} < \infty$.
\end{enumerate}   
\end{definition}

The second condition will be vital for the applicability of the
\emph{Chinese remainder theorem} (CRT) in our arguments.  Note that
the third condition is also known as \emph{thinness} in the literature
\cite{Keller-2,DKKL}.  Let $\cB$ be of Erd\H{o}s type, and
$V^{}_{\cB}$ be the corresponding $\cB$-free set. When $2\in\cB$, we
are in the degenerate situation that $V^{}_{\cB}$ lies in one coset
modulo $2$, which requires additional arguments and leads to a
different result, as we shall see.

Since we always assume that $\cB$ is fixed, we write $V$ instead of
$V^{}_{\cB}$ from now on. The left translate of $V$ by an integer $t$
is the set $V\! -t = \{ m-t : m\in V \}$, and we define
\begin{equation}\label{eq:def-X}
    \XX \, = \, \overline{ \{ V \! - t : t \in \ZZ \} } 
\end{equation}
as the orbit closure under the action of $\ZZ$, where the closure is
taken in the \emph{local topology}. Here, two subsets of $\ZZ$ are
$\frac{1}{n}$-close if they exactly agree on $[-n,n]$. The set $\XX$
is compact, with continuous action of $\ZZ$ on it by translation, so
$(\XX, \ZZ)$ is a TDS.

An alternative view emerges via an isomorphism to a symbolic subshift
of $\{ 0, 1\}^{\ZZ}$. Identifying $V$ with its characteristic
function, $1^{}_{V}$, the translation of $V$ by $1$ to the left
corresponds to the action of $S$ (the left shift) on $x = 1^{}_{V}$,
and $\XX$ becomes a subshift. The local topology used in
\eqref{eq:def-X} corresponds to the standard product topology of
$\{ 0, 1\}^{\ZZ}$.  We identify any subset $U\subseteq\ZZ$ with its
characteristic function $1^{}_{U}$. This way, the extremal subsets
$\ZZ$ and $\varnothing$ are tacitly seen as the elements $\bs{1}$ and
$\bs{0}$ in $\{ 0, 1\}^{\ZZ}$, respectively.  The analogous
correspondence exists for finite subsets in a given range. We shall
use both representations in parallel, and switch between them without
further notice to profit from the symbolic and the geometric side as
much as possible. Consistent with this, we shall use $S$ both for the
left shift (symbolically) and the left translation (geometrically),
where $S(U) = U\nts -1$ for any $U\subseteq \ZZ$.

When $\cB$ is of Erd\H{o}s type, we also call the corresponding
$\cB$-free shift Erd\H{o}s.  Once $\cB$ is given, not necessarily
Erd\H{o}s but at least with an infinite subset of pairwise coprime
integers, another interesting shift is
\[
  \AAA \, \defeq \, \{ U \subset \ZZ : \card (U \bmod b)
  \leqslant b-1 \text{ holds for all } b\in\cB \} \ts .
\]
Its members (or elements) are called \emph{admissible} sets.
One clearly has $\XX \subseteq \AAA$, because $V \in \AAA$, which
gives rise to the following property.

\begin{lemma}\label{lem:no-periods}
  If\/ $\XX$ is a\/ $\cB$-free shift, for a\/ $\cB$ with an infinite
  pairwise coprime subset, the only element of\/ $\XX$ with a
  non-trivial period is\/ $\varnothing$, so all non-empty\/ $U\in \XX$
  are non-periodic.
\end{lemma}

\begin{proof}
  Clearly, every $\ZZ$-translation maps $\varnothing$ onto itself, so
  consider any non-empty $U \in \XX$. Then, we must have $m\in U$ for
  some $m\in\ZZ$. Assume that $U \nts + t = U$ for some
  $0\ne t \in \ZZ$, where we may take $t>0$ without loss of
  generality. We note that $1^{}_{U}$ is then the $t$-periodic
  repetition of the patch $[u^{}_{0}, u^{}_{1}, \ldots , u^{}_{t-1}]$,
  and at least one of the $u_i$ must equal $1$.

  Select some $b\in\cB$ with $\gcd (b,t)=1$, which exists under our
  assumptions.  Then, the arithmetic progression
  $\{ m + \ell \ts t : \ell \in \ZZ \}$ populates \emph{all} cosets
  modulo $b$. This implies that $U \notin \AAA \supseteq \XX$, which
  is a contradiction. Consequently, we must have $t=0$ as claimed.
\end{proof}

When $\cB$ is Erd\H{o}s, one actually has $\XX = \AAA$, which is the
origin of many simplifying properties.  Let us recall some of them.

\begin{prop}
  Let\/ $\XX$ be an Erd\H{o}s $\cB$-free subshift, with defining set\/
  $V\nts$. Then, the following properties hold.
\begin{enumerate}\itemsep=2pt
 \item $V$ has a dense orbit in\/ $\XX$, by construction.
 \item $V$ is a maximal element of\/ $\XX$, so no point can be added
   to\/ $V\nts $.
 \item $\XX = \AAA$, so any admissible set is automatically in\/ $\XX$. 
 \item $\XX$ is hereditary, that is, for every\/ $U\in\XX$, one has\/
   $U' \in \XX$ for all\/ $U' \subseteq U$.
 \item $h_{\textnormal{top}}(\XX)=\log(2)\prod_{b\in\cB}
   \left(1-\frac{1}{b}\right)>0$.  
\end{enumerate}  
\end{prop}

\begin{proof}
  (1) is clear, and $(\XX, \ZZ)$ is an example of single orbit
  dynamics \cite{Weiss}. Property (2) follows from the definition of
  $V\nts $, as any added point must occupy the $0$ coset modulo some
  $b\in\cB$.  For the proof of Property~(3), see
  \cite[Cor.~4.2]{Abda}, while (4) is a direct consequence of (3).
  The formula for the topological entropy in (5) ultimately follows
  from the existence of the density of $V$ (relative to an averaging
  sequence of growing intervals placed symmetrically around $0$),
  which is given by
  $\dens(V)=\prod_{b\in\cB}\left(1-\frac{1}{b}\right)$, in conjunction
  with the interpretation of $V$ as a weak model set with maximal
  density; compare \cite[Thm.~5.1]{BBN} and \cite[Sec.~5.2]{HR}.
\end{proof}

\begin{remark}\label{rem:her-admin}
  For a general set $\cB$ and its shift, $\XX$, one can define the
  smallest hereditary subshift that contains $\XX$, denoted by
  $\widetilde{\XX}$. Then, one has the inclusions
\[
   \XX \, \subseteq \, \widetilde{\XX} \, \subseteq \, \AAA \ts ,
\]
which are generally strict. In \cite[Thm.~3]{Keller-2}, it was shown
that, if $\cB$ is \emph{taut} and contains an infinite pairwise
coprime subset, one has $\XX= \widetilde{\XX}$ and $\XX$ is then
hereditary; see \cite{DKKL} for the definition of tautness and
\cite{KKL, Keller-2} for equivalent characterisations. On the other
hand, one can construct examples of sets $\cB$ that are taut and have
an infinite pairwise coprime subset, but still satisfy
$\widetilde{\XX}\subsetneq \AAA$; see \cite[Ex.~2.47]{DKKL}. Since
both hereditariness and admissibility are central to our proofs, we
stick to the somewhat more restrictive Erd\H{o}s condition of
Definition~\ref{def:simple}.  Clearly, some results can be extended to
more general admissible shifts. \exend
\end{remark}

As a general strategy, we construct admissible patterns whose images
under certain mappings are no longer admissible, thus providing
restrictions on the mappings in question. Here, we exploit the
equality of the three shifts in Remark~\ref{rem:her-admin} and the
coprimality condition in Definition~\ref{def:simple} (by invoking the
CRT) to construct finite patterns, whose points can be chosen to be as
sparse as necessary while remaining admissible. This ensures that the
relevant block codes only pick up information from the finite patterns
with the pre-set modular properties we have chosen. Hereditariness and
admissibility also allow one to work with the actions of mappings on
\emph{singleton} sets, and the \emph{a posteriori} extension to the
whole shift space $\XX$.

\section{Centraliser, full centraliser, and full normaliser}

To characterise a TDS, it is useful to determine its symmetries and
reversing symmetries. Here, due to our one-dimensional setting, the
latter part is essentially trivial, as every $\cB$-free shift is
inversion symmetric, so the reversing symmetry group \cite{Baa, Ian}
is always of the form $\cR = \cS \rtimes C_2$, where
$\cS = \cent^{}_{\cH} (\langle S \ts \rangle)$ and $C_2$ is generated
by inversion.  For $\cS$, Mentzen~\cite{Mentzen} proved the following
result; see also \cite{Keller, Eden} for proofs for larger classes.

\begin{theorem}
  The centraliser of an Erd\H{o}s $\cB$-free shift is trivial, so\/
  $\cS = \langle S \ts \rangle \simeq \ZZ$.  \qed
\end{theorem}

This is a situation otherwise known from certain low-complexity
shifts, and indicates that the centraliser does not explore much of
the TDS. As a consequence, one can look at the $k$-symmetries
\cite{Lamb-2, QL}, as defined by
\[ 
  \cC^{(k)}_{\cH} \, \defeq \, \cent^{}_{\cH}
  ( \langle S^k \ts \rangle ) \ts ,
\]
and then at the union of these groups, 
\[
    \cC^{(\infty)}_{\cH} \, \defeq \bigcup_{k\in\NN} \cC^{(k)} .
\]
This was introduced for one-dimensional subshifts in \cite{HKS}, where
it is called the \emph{stabilised automorphism group}.  This group can
be larger than $\cS = \cC^{(1)}_{\cH}$, which reveals further
structure; see \cite{HKS} and \cite{Schmieding} in the case of sofic
shifts and mixing shifts of finite type, and \cite{Jones-Baro} for
Toeplitz shifts and odometers. However, it can still be trivial, as is
the case for Sturmian shifts; see \cite[Ex.~3.3]{HKS}.  It is our goal
to determine $\cC^{(\infty)}_{\cH}$ for Erd\H{o}s $\cB$-free shifts.

To this end, one has to find all elements $H\in\cH$ that commute with
$S^k$ for some $k\in\NN$, but possibly not with smaller powers of $S$,
where $k\in\NN$ means no restriction because $H$ commutes with $S^k$
if and only if it commutes with $S^{-k}$.  Any $H\in\cH$ of this kind
implies that $S^{}_{\nts H} = H S H^{-1}$ satisfies
$S_{\nts H}^{\ts k} = S^{k}$, so we potentially get additional
$k\ts $th roots of $S^k$ beyond $S$.

\begin{fact}\label{fact:fixed-point}
  Let\/ the assumptions be as in Lemma~$\ref{lem:no-periods}$.
  If\/ $H\in \cH$ commutes with\/ $S^k$ for some\/ $k\in\NN$, it fixes
  the empty set, so we have\/ $H (\varnothing) = \varnothing$.
\end{fact}

\begin{proof}
  Assume $H(\varnothing) = U$. Since $S^k (\varnothing) = \varnothing$, 
  we then obtain
\[
    U \, = \, H (\varnothing) \, = \, H \bigl( S^k (\varnothing) \bigr) 
    \, = \, S^k \bigl( H (\varnothing) \bigr) \, = \, S^k (U) \ts .
\]  
The set $U$ thus satisfies $U \nts - k = U$, which implies
$U=\varnothing$ by Lemma~\ref{lem:no-periods}.
\end{proof}

Assume that $H\in\cH$ commutes with $S^k$ for some $k\in\NN$, but
possibly not with smaller powers of $S$. Then, the
Curtis--Lyndon--Hedlund theorem has a modular variant in the sense
that $H$ is characterised by $k$ block maps,
$\varphi^{}_{m \bmod k}\colon \cL_{2\rho+1}\to \{0,1\}$, for
$0\leqslant m \leqslant k-1$ and some $\rho \in \NN_0$; compare
\cite[Lemma~3.2]{HKS}. Here, $\cL_{2\rho+1}$ is the set of finite
admissible patterns of length $2\rho+1$. In principle, the radius
could depend on $m$, but simply using the maximum of these radii means
no loss of generality. Since we will be working with point sets
throughout, let us state the following immediate but important
consequence.

\begin{fact}\label{fact:sep}
  Let\/ $H$ be a continuous mapping of the full shift\/
  $\{ 0,1 \}^{\ZZ}$ into itself that is specified by one or several
  block maps, all of radius\/ $\leqslant \rho$, with
  $H(\bs{0})=\bs{0}$. Then, the corresponding mapping on the subsets
  of\/ $\ZZ$, also called\/ $H$, satisfies\/
  $H(\varnothing) = \varnothing$ and has the following separation
  property. If\/ $U\subset \ZZ$ satisfies\/ $U = V \cup W$ such that\/
  $V$ and\/ $W$ have distance larger than\/ $2\rho$, we have\/
  $H(U) = H(V) \cup H(W)$.  \qed
\end{fact}

\smallskip

From now on, for the remainder of the paper, we always assume that the
set $\cB$ is Erd\H{o}s in the sense of Definition~\ref{def:simple}.

\begin{lemma}\label{lem:trans-1}
  Let\/ $H\in \cH$ commute with\/ $S^k$, for some\/ $k\in\NN$. Then,
  for any\/ $0\leqslant m \leqslant k-1$, the mapping\/ $H$ acts as a 
  translation on the singleton set\/ $U_m \defeq \{ m \}$, and hence also 
  on all\/ $U_{m+ \ell k}$ with\/ $\ell\in\ZZ$.
\end{lemma}

\begin{proof}
  Consider $U_m$ for a fixed $m$ between $0$ and $k-1$. Since
  $H(U_m) \ne \varnothing$ as a consequence of
  Fact~\ref{fact:fixed-point}, the image must contain some
  $U_{\ell}$. If there is only one singleton set in the image, $H$
  acts as a translation. Otherwise, we have
  $U_{\ell} \cup U_{\ell'} \subseteq H(U_m )$ for some $\ell <
  \ell'$. Select a number $c \in\cB$ that is coprime with $k$ and
  satisfies $\ell' \nts - \ell \not\equiv 0 \bmod c$, which is clearly
  possible.

  By assumption, we have $H( U_{m+k} ) = H( U_m) + k$.  Now, select
  integers $\{ s^{}_{1}, \ldots , s^{}_{c-1} \}$ with
\begin{itemize}\itemsep=2pt
\item $s_{i} \equiv i \bmod c$ for all $1\leqslant i \leqslant c-1$;
\item $s_{i} \equiv 0 \bmod b$ for every $b\in\cB$ with $b < c$
    and $\gcd (b,k)=1$;
\item $s_{i} \equiv m \bmod k$.
\end{itemize}
Since all moduli occurring in the congruence conditions are mutually
coprime, such numbers exist by the CRT. The first condition ensures
that all cosets mod $c$ except $0$ are occupied. The second restricts
the number of cosets modulo small $b$ that are coprime with $k$ to
one.  The third condition makes sure that all $s_i$ are at positions
where the same block map applies, namely $\varphi^{}_{m}$.  From this
condition, we also get $s_i \equiv m \bmod \gcd(b,k)$ for all
$b\in\cB$ that are not coprime with $k$. Consequently, for any such
$b$, the number of occupied cosets mod $b$ is at most
$\tfrac{b}{\gcd(b,k)}$ and thus less than $b$. For all remaining
$b\in\cB$, which then satisfy $b > c$, the $c-1$ points can at most
occupy $c-1 < b$ cosets mod $b$. Consequently,
$\{ s^{}_{1} , \ldots , s^{}_{c-1} \}$ is an admissible set, that is,
an element of $\AAA$.

Moreover, since there are infinitely many solutions, we can select
the points to be well separated, that is, with
$\lvert s_i - s_j \rvert > 2 \rho$ for all $i\ne j$. Then,
\[
  H \bigl( U^{}_{s^{}_{1}} \cup \cdots \cup U^{}_{s^{}_{c-1}}
  \bigr) \; = \bigcup_{1\leqslant i \leqslant c-1}\! H (U^{}_{s_i}) \ts ,
\]
where we used Fact~\ref{fact:sep}, and this image contains $c-1$
cosets, and at least one more because
$\ell'\nts - \ell \not\equiv 0 \bmod c$. Thus, the image is not
admissible, hence not in $\XX$, which contradicts the assumption
$H\in\cH$. Consequently, we have
$H (U_m) = U^{}_{0} + t(m) = U^{}_{t(m)}$ for
$0\leqslant m \leqslant k-1$. The final claim is clear because $H$
commutes with $S^k$.
\end{proof}

\begin{fact}\label{fact:trans-2}
  Given\/ $H$ and\/ $k$ as in Lemma~$\ref{lem:trans-1}$, let\/ $t(m)$
  be the positions from its proof. Then, no two of these
  positions can be congruent\/ $\bmod \, k$.
\end{fact}

\begin{proof}
  Assume $H(U_m) = U_s$ and $H(U_n) = U_s + \ell k$, with
  $0 \leqslant m,n < k$.  Then, since $S^k$ also commutes with
  $H^{-1}$, we get
\[
  U_n \, = \, H^{-1} (U_s + \ell k) \, = \,
  H^{-1} (U_s) + \ell k \, = \, H^{-1} \bigl( H(U_m)\bigr)
  + \ell k \, = \, U_m + \ell k \ts ,
\]
which implies $n = m + \ell k$. With $0 \leqslant m,n < k$, this
is impossible for $m \ne n$.
\end{proof}

The next step will be to understand which sets of positions $t(m)$,
and hence translates, are possible, where we start with the following
subtlety for the case $2\in\cB$.

\begin{example}\label{ex:two-is-special}
  If $\cB$ is Erd\H{o}s, but with $2\in\cB$, we get $V\subset 2\ZZ+1$
  and a partition $\XX = \XX_0 \, \dot{\cup}\, \XX_1$ into two parts,
  with $\XX_0$ comprising all elements of $\XX$ with points only on
  even positions and $\XX_1 = S(\XX_0)$.  Now, define a mapping $H$ by
  $H\big|_{\XX_0} = \id$ and $H\big|_{\XX_1} = S^t$, which is an
  element of $\cH$ for any $t\in 2\ZZ$. Clearly, $H$ always commutes
  with $S^2$, but with $S$ only for $t=0$. When $t\ne 0$, the mapping
  $SH$ interchanges $\XX_0$ and $\XX_1$. So, we get non-trivial
  $2$-symmetries, for a trivial reason (the decomposition of $\XX$
  into two parts that live on the cosets modulo $2$). Other examples
  can be obtained as $S^{\ell} H$ with $\ell\in\ZZ$, and
  $\cC^{(2m)}_{\ts \cH} (\langle S \ts \rangle)$ contains another
  group isomorphic with $\ZZ$, for every $m\in\NN$.  \exend
\end{example}

\begin{remark}
  Note that a crucial element in the previous example is the
  $S$-cyclic partition of $\mathbb{X}$. We suspect that, in the more
  general setting of non-Erd\H{o}s $\cB$-free shifts (also for
  $\cB$-free lattice systems or $\cB$-free numbers in general number
  fields), this provides a mechanism for non-trivial $k$-symmetries to
  exist. For an example of a non-Erd\H{o}s $\cB$-free shift with
  non-trivial centraliser (which also happens to be a Toeplitz
  subshift), see \cite[Sec.~3.4]{DKK}.  \exend
\end{remark}

\begin{lemma}\label{lem:trans-3}
  Let\/ $\cB$ be Erd\H{o}s and let\/ $H$ and\/ $k$ be as in
  Lemma~$\ref{lem:trans-1}$. If\/ $2 \notin\cB$ or if\/ $k$ is odd,
  the mapping\/ $H$ acts as a uniform translation on the singleton
  sets\/ $U_m$, that is, for a fixed\/ $t\in\ZZ$, $H (U_m) = U_{m+t}$
  holds for all\/ $m\in\ZZ$.

  When\/ $2\in\cB$ and\/ $k$ is even, $H$ acts as a uniform
  translation on all\/ $U_{2m}$ and as one on all\/ $U_{2m+1}$, where
  the difference of these two translations may be any even integer.
\end{lemma}

\begin{proof}
  Assume otherwise, so consider $k>1$ together with
  $H(U^{}_{0}) = U_s$ and $H( U_m) = U_{s+m+\ell}$ for some
  $1\leqslant m < k$. If we show that only $\ell=0$ is possible, for
  any such $m$, our claim also follows for every pair $U_m$ and
  $U_{m'}$.  Assume $\ell\ne 0$.
  
  Fix a $2 < c\in\cB$ with $\gcd (c,k)=1$ and
  $\ell \not\equiv 0 \bmod c$, and select integers
  $\{ s^{}_{1}, \ldots , s^{}_{c-1} \}$ with
\begin{itemize}\itemsep=2pt
\item[(i)] $s_{i} \equiv i \bmod c$ for $1 \leqslant i \leqslant c-1$;
\item[(ii)] $s_{i} \equiv 0 \bmod k$ for all $i$ and 
    $\lvert s_{i} - s_{j} \rvert > 2 \rho $ for all $i\ne j$;
  \item[(iii)] $s_{i} \equiv 0 \bmod b$ for every $b\in\cB$ with
    $b < c$ and $\gcd(b,k)=1$.
\end{itemize}
This is possible by the CRT. Further, select an integer
$ s^{\prime}$ with
\begin{itemize}\itemsep=2pt
\item[(i$'$)] $s^{\prime} \equiv -\ell \bmod c\ts $;
\item[(ii$'$)] $s^{\prime} \equiv m \bmod k$ and 
   $\lvert s^{\prime} - s^{}_{i}\rvert > 2 \rho $ for all $i$.
 \item[(iii$'$)] $s^{\prime} \equiv 0 \bmod b$ for every $b\in\cB$
   with $b < c$ and $\gcd(b,k)=1$.
\end{itemize}

Again, this is possible by the CRT, and we then see that the union
\[
   U \, = \, U_{s^{\prime}} \; \cup \; \bigcup_{i=1}^{c-1} U_{s^{}_{i}}   
\] 
is still admissible when $k$ is odd or when $2\notin\cB$. Indeed, for
every $b\in\cB$ with $b>c$, admissibility is satisfied because
$\card (U) = c < b$. For $b=c$, observe that
$\{s^{}_{1}, \ldots , s^{}_{c-1}\}$ populates all cosets mod $c$
except $0$, while $s'$ occupies one of them a second time, as a result
of (i) and (i$'$) together with $\ell\not\equiv 0\bmod c$. Next, for
$b\in\cB$ with $b<c$ and $\gcd(b,k)=1$, only the zero coset modulo $b$
is occupied, which is fine.

So, it remains to check the numbers $b\in\cB$ with $b<c$ and
$\gcd(b,k)>1$. Note that $b\ne 2$ since we assumed that $2\notin\cB$
or that $k$ is odd. Then, by (ii) and (ii$'$), we get $s_i \equiv 0 $
and $s' \equiv m \bmod \gcd(b,k)$.  Consequently, since $b>2$ now, the
number of occupied cosets is at most
$\tfrac{b}{\gcd(b,k)} + 1 \leqslant \tfrac{b}{2} + 1 < b$, which is
what we need.

However, applying Fact~\ref{fact:sep} again, the image under $H$
occupies all cosets $\ne s \bmod c$ (coming from
$\bigcup_{i} U_{s_i}$) and the coset $s + s' + \ell \equiv s \bmod c$
(from the image of $U_{s'}$). Together, they occupy \emph{all} cosets
modulo $c$, which is impossible for $H\in\cH$. So, assuming
$\ell\ne 0$ leads to a contradiction, and we must have $\ell=0$,
proving the claim for $2\notin\cB$ or $k$ odd.

In the remaining case, $2\in\cB$ with $k$ even, we saw in
Example~\ref{ex:two-is-special} what is possible. The above argument
can still be used when $m$ is even, which then implies that we have to
deal with \emph{two} translations, namely one for $U_m$ with $m$ even
and another with $m$ odd. Without loss of generality, after replacing
$H$ by $S^t H$ with suitable $t$, we may assume $H(U_{2m})=U_{2m}$
together with $H(U_{2m+1}) = U_{2m+1} + t'$ for $m\in\ZZ$, where $t'$
must be an even integer because $H$ is invertible. Clearly, all
$t'\in2\ZZ$ are possible.
\end{proof}

\begin{remark}\label{rem:more-on-two}
  In view of Lemma~\ref{lem:trans-3} and
  Example~\ref{ex:two-is-special}, one can think of the case
  $2\in\mathcal{B}$ as a somewhat degenerate (or reducible) case,
  which we keep in the background for the rest of the paper for ease of
  presentation.

  Let us also note that one can prove Lemma~\ref{lem:trans-3} without
  Lemma~\ref{lem:trans-1} and Fact~\ref{fact:trans-2}, by assuming
  $H(U_0)\supseteq U_s$ and $H(U_m)\supseteq U_{s+m+\ell}$ and
  admitting any $m$, where the case $m=0$ then shows that $H$ sends
  $U_0$ to a singleton set. We chose the above route for pedagogic
  reasons and its resemblance to previous arguments in \cite{Mentzen,
    BBHLN}.  \exend
\end{remark}

Let $H$ and $k$ be as in Lemma~\ref{lem:trans-1}, assuming
$2\notin\cB$, and let $t$ be the uniform translation from
Lemma~\ref{lem:trans-3}. Then, $H' \defeq S^t H$ lies in $\cH$ and
still commutes with $S^k$, with the simplification that
$H' (U_m) = U_m$ holds for all $m\in\ZZ$. We now rename it as $H$ and
then assume (without loss of generality) that $H$ fixes all singleton
sets. We now need to show that $H = \id$.

As used above, $H$ is characterised by $k$ block codes, with common
radius $\rho$, and we now know that
$\varphi^{}_m (0, \ldots, 0,1, 0,\ldots , 0 ) = 1$ holds for all
$0\leqslant m \leqslant k-1$. Let $\psi$ be any of them.

\begin{lemma}\label{lem:no-extra}
  Let\/ $\cB$ be Erd\H{o}s, with\/ $2\notin \cB$.  Let\/ $H\in\cH$
  and\/ $k$ be as above, where\/ $H$ fixes all singleton sets, and
  let\/ $\psi = \varphi^{}_m$ be any of the block codes that define\/
  $H$, with\/ $0\leqslant m \leqslant k-1$.  Then, it satisfies\/
  $\psi (*,\ldots,*,0,*, \ldots, *) = 0$, where each\/ $*$ can be\/
  $0$ or\/ $1$, subject to the condition that the entire block of
  length\/ $2 \rho + 1$ is admissible.
\end{lemma}

\begin{proof}
  Recall that $\psi (0, \ldots , 0) = 0$ as a consequence of
  Fact~\ref{fact:fixed-point}, where $(0, \ldots , 0)$ corresponds to
  an empty stretch of length $2\rho+1$ in the point set formulation.
  There is nothing to show when $\rho=0$, so assume
  $\rho \geqslant 1$. Now, to the contrary of our claim, assume that
  $\psi (x, 0, y) = 1$, where at least one position in $x$ or $y$ is
  $1$. This code (of length $2\rho+1$) corresponds to a non-empty,
  finite set $U^*$ with $0\notin U^*$ and no entry larger than $\rho$
  in absolute value. Concretely, let $0\ne \ell\in U^*$, and observe
  that $U^*$ contains at most $2\rho$ elements.
   
  Now, choose some $c\in\cB$ with $c > 2\rho$, but not necessarily
  coprime with $k$. Then, select integers
  $\{ s^{}_{1}, \ldots , s^{}_{c-1} \}$ that are mutually separated,
  well separated from $U^*$, and occupy all cosets $\ne 0 \bmod c$
  such that $U = U^* \cup \{ s^{}_{1}, \ldots , s^{}_{c-1} \}$ is
  still admissible.  This is possible, by the CRT, based on arguments
  similar to the ones used before. Here, we demand
\begin{itemize}\itemsep=2pt
  \item $s_i \equiv i \bmod c$ for all $1\leqslant i \leqslant c-1$;
  \item $s_i \equiv \ell \bmod b$ for all $b\in\cB$ with
     $c \ne b \leqslant c + 2 \rho$;
  \item $\lvert s_i \rvert > 2 \rho$ for all $i$ and
     $\lvert s_i - s_j\rvert > 2 \rho$ for all $i \ne j$.
\end{itemize}
  Note that no congruence condition modulo $k$ is required, because
  $H$ fixes all singleton sets, so
  $\varphi^{}_{m} (0, \ldots, 0, 1, 0, \ldots , 0) = 1$ holds for all
  $0\leqslant m \leqslant k-1$.

  The set $U$ is admissible by design (which here also means that
  there is some $t \in \ZZ$ such that $V\! + t$ agrees with $U$ on an
  interval containing $U$, and even in such a way that all other
  points of $V\! +t$ have distance $> 2 \rho$ from $U$). It is
  constructed relative to the point $0$, thus for the block map
  $\varphi^{}_{0}$ acting at position $0$, which we now need to shift
  to $m$, so that the action of $H$ at $m$ is controlled by the block
  map $\psi = \varphi^{}_{m}$.

  We thus consider $U\nts + m$, which is still admissible. Then,
  $H(U \nts +m)$ must contain all elements from
  $m + \{ s^{}_{1}, \ldots , s^{}_{c-1} \}$, due to the behaviour of
  $H$ on well-separated singleton sets, and the point $m$ in addition
  (which comes from the image of $U^{*}\nts +m$), due to our
  assumption on the code $\psi = \varphi^{}_m$.  Together, this is not
  admissible, so $H(U\nts +m) \notin\XX$, which is impossible for
  $H\in\cH$.
\end{proof}

This lemma implies that the special $H$ under consideration cannot
turn any $0$ into a $1$, hence cannot create any new point. This has
the following immediate consequence, which corresponds to the
monotonicity considered in \cite{Eden}.

\begin{prop}\label{prop:subset}
  Let\/ $\cB$ be Erd\H{o}s, with\/ $2\notin\cB$.  If\/ $H\in\cH$ fixes
  singleton sets and commutes with\/ $S^k$, for some\/ $0\ne k\in\ZZ$,
  it satisfies\/ $H (U) \subseteq U$ for all\/ $U\in\XX$.  \qed
\end{prop}

Consider $H (U) \subseteq U$.  Clearly, $H^{-1}$ is another element of
$\cH$ that commutes with $S^k$, so we may conclude from
Lemma~\ref{lem:trans-3} that there is a fixed $t'\in\ZZ$ such that
$H^{-1} (U_m) = U_m + t'$ holds for all $m\in\ZZ$. This implies
\[
  U_m + t' \, = \, H^{-1} (U_m) \, = \, H^{-1} \bigl( H(U_m)
  \bigr) \, = \, U_m \ts ,
\]
which means $t' \nts =0$ and hence $H^{-1} (U) \subseteq U$ by
Proposition~\ref{prop:subset}, as well as $U \subseteq
H(U)$. Consequently, $U \subseteq H (U) \subseteq U$ and thus
$H(U) = U$ for all $U\in \XX$, which means $H = \id$.

Returning to our original $H$, we conclude that it was some power of
$S$, which applies to \emph{any} $H\in\cH$ that commutes with $S^k$
for some $k\in\NN$. We have arrived at our main result.

\begin{theorem}\label{thm:full-cent}
  If\/ $\XX$ is an Erd\H{o}s\/ $\cB$-free shift, with\/ $2\notin\cB$,
  its centraliser is trivial, as is the full centraliser, so we have\/
  $\, \cent^{(\infty)}_{\cH} (\langle S \ts \rangle ) = \cent^{}_{\cH}
  ( \langle S \ts \rangle ) = \langle S \ts \rangle \simeq \ZZ$.  \qed
\end{theorem}

Let us revisit Example~\ref{ex:two-is-special}. When $k\in\NN$ is odd,
Lemma~\ref{lem:trans-3} applies, wherefore we obtain
$\cent^{(k)}_{\cH} (\langle S \ts \rangle) = \langle S \ts \rangle$
for any such $k$. When $2\in\cB$ and $k$ is even, we have to deal with
two translations, as also covered by Lemma~\ref{lem:trans-3}, where we
may assume $H (U_m) = U_m$ for all even $m$ without loss of
generality. Then, one needs to run the arguments from the proof of
Lemma~\ref{lem:no-extra} twice, namely for $\ell$ even or odd,
respectively. This leads to a variant of
Proposition~\ref{prop:subset}, where $H(U)\subseteq U$ for all
$U\in\XX_0$ and $H(U) \subseteq U+t$ for all $U\in\XX_1$. Its inverse
must then satisfy $H^{-1}(U)\subseteq U$ for $U\in\XX_0$ and
$H^{-1}(U) \subseteq U-t$ for $U\in\XX_1$. This gives
$H\big|_{\XX_0} = \id$ together with $H\big|_{\XX_1} = S^{-t}$ for an
even $t$, which holds for every even value of $k$.  Let $H_2$ denote
this mapping for $t=2$; compare Example~\ref{ex:two-is-special}. Then,
we have
$\cent^{(k)}_{\cH} (\langle S \ts \rangle) = \langle H_2\rangle \times
\langle S \ts \rangle \simeq \ZZ \times \ZZ$, for any even
$k\in\NN$. This can be summarised as follows.

\begin{coro}
  Let\/ $\cB$ be Erd\H{o}s, with\/ $2 \in\cB$. Then, there are no\/
  genuine $k$-symmetries when\/ $k$ is odd, while\/ $k$ even admits
  the ones from Example~$\ref{ex:two-is-special}$. The centraliser is
  still trivial, but the full centraliser is\/
  $\,\cC^{(\infty)}_{\cH} (\langle S \ts \rangle ) \simeq
  \ZZ\times\ZZ$.  \qed
\end{coro}

Here, we also know that
$\norm^{}_{\cH} (\langle S \ts \rangle) = \cent^{}_{\cH} (\langle S
\ts \rangle ) \rtimes C_2$ with $C_2 = \langle R \ts \rangle$, where
$R$ is the reflection in the origin. Since Theorem~\ref{thm:full-cent}
in particular establishes the absence of any non-trivial (or genuine)
$k$-symmetry when $2\notin \cB$, we may conclude that there are also
no genuine $k$-reversing symmetries, compare \cite{Lamb-2,Lamb-3} for
background, and we thus have the following.

\begin{coro}\label{coro:main}
  Erd\H{o}s\/ $\cB$-free shifts are reversible, with trivial
  centraliser. When\/ $2\notin\cB$, the full group of
  $($reversing$\ts\ts )$ symmetries contains no further element,
  so
\[
   \pushQED{\qed}
   \norm^{(\infty)\vphantom{t^{\hat{I}}}}_{\cH} (\langle S \ts \rangle )
   \, = \, \norm^{}_{\cH} ( \langle S \ts \rangle ) \, = \,
   \langle S \ts \rangle \rtimes \langle R \ts \rangle
   \, \simeq \, \ZZ \rtimes C_2 \ts .
\qedhere \popQED
\]
\end{coro}

Using Theorem~\ref{thm:full-cent}, one can easily construct pairs of
Erd\H{o}s $\cB$-free shifts with (trivially) isomorphic full
centralisers whose topological entropies are rationally
\emph{independent}. More concretely, choosing
$\cB=\left\{p^2 : \,p \text{ prime} \right\}$ and
$\cB^{\prime}=\left\{p^4 : \,p \text{ prime} \right\}$,
Theorem~\ref{thm:full-cent} implies that they have isomorphic full
centralisers but
\begin{equation}\label{eq:not-rat}
  \frac{h_{\text{top}}(\XX)}{h_{\text{top}}(\XX^{\prime})} \, = \,
  \frac{\zeta(4)}{\zeta(2)} \, = \, \frac{\pi^2}{15}\notin
  \QQ \ts .
\end{equation}
Here, $\XX$ and $\XX^{\prime}$ are the shifts corresponding to $\cB$
and $\cB^{\prime}$, and $\zeta$ is Riemann's zeta function.

Clearly, we expect similar results to hold for suitable $\cB$-free
lattice systems, including power-free integers in algebraic number
fields (then turned into symbolic dynamical systems via the Minkowski
embedding). However, the entire concept of (extended) $k$-symmetries
requires substantial care, and is thus a topic for future
investigations.

\bigskip

\section*{Acknowledgements}

It is our pleasure to thank Fabian Gundlach and Stanis{\l}aw Kasjan
for helpful discussions and \'{A}lvaro Bustos and Jan Maz\'{a}\v{c}
for useful hints on the manuscript. This work was supported by the
German Research Foundation (DFG, Deutsche Forschungsgemeinschaft)
within the CRC TRR 358/1 (2023) -- 491392403 (Bielefeld -- Paderborn).

\end{document}